\documentclass[twoside]{amsart}
\usepackage{latexsym}
\usepackage{amssymb,amsmath,amsopn}
\usepackage[dvips]{graphicx}   
\usepackage{color,epsfig}      

\newenvironment{myquote}[2]%
               {\begin{list}{}{\leftmargin#1\rightmargin#2}\item{}}%
               {\end{list}}
               
\newtheorem{thm}{Theorem}
\newtheorem{lem}{Lemma}
\theoremstyle{definition}
\newtheorem{defn}{Definition} 
\newtheorem{rem}{Remark}
\newtheorem{ques}{Question}

\newtheorem{cor}{Corollary}
\renewcommand{\Re}{\mathbb R}
\newcommand{\Eu}{\mathbb E}

\newcommand{\Z}{\mathbb Z}
\newcommand{\Q}{\mathbb Q}
\def\bea{\begin{eqnarray}}
\def\eea{\end{eqnarray}}

\DeclareMathOperator{\relint}{relint}
\DeclareMathOperator{\Err}{Err}
\DeclareMathOperator{\inter}{int}

\DeclareMathOperator{\conv}{conv}
\DeclareMathOperator{\card}{card}

\DeclareMathOperator{\area}{area}

\DeclareMathOperator{\fraction}{frac}
\DeclareMathOperator{\sign}{sign}

\begin{document}
\title[Equilibrium points]{On the equilibria  of finely discretized curves and surfaces}
\author[G. Domokos, Z. L\'angi \and T. Szab\'o]{G\'abor Domokos, Zsolt L\'angi \and T\'\i mea Szab\'o}

\address{G\'abor Domokos, Dept. of Mechanics, Materials and Structures, Budapest University of Technology,
M\H uegyetem rakpart 1-3., Budapest, Hungary, 1111}
\email{domokos@iit.bme.hu}
\address{Zsolt L\'angi, Dept. of Geometry, Budapest University of Technology,
Egry J\'ozsef u. 1., Budapest, Hungary, 1111}
\email{zlangi@math.bme.hu}
\address{T\'\i mea Szab\'o, Dept. of Mechanics, Materials and Structures, Budapest University of Technology,
M\H uegyetem rakpart 1-3., Budapest, Hungary, 1111}
\email{tszabo@szt.bme.hu}

\subjclass{53A05, 53Z05}
\keywords{equilibrium, convex surface, Poincar\'e-Hopf formula, polyhedral approximation.}

\begin{abstract}
Our goal is to identify the type and number of static equilibrium points 
of solids arising from fine, equidistant $n$-discretrizations of smooth, convex surfaces.
We assume uniform gravity and a frictionless, horizontal, planar support. 
We show that as $n$ approaches  infinity these numbers fluctuate around specific values which
we call the imaginary equilibrium indices associated with the approximated smooth surface.
We derive simple formulae for these numbers in terms of the principal curvatures
and the radial distances of the equilibrium points of the solid from its center of gravity.
Our results are illustrated on a discretized ellipsoid and match well the observations on natural pebble
surfaces. \end{abstract}
\maketitle

\section{Introduction}

The study of  equilibria of rigid bodies was initiated by Archimedes
\cite{Archimedes1}; his results have been used in naval design even in the 18th century (cf. \cite{Archimedes2}).
Archimedes' main concern was the number of the stable balance points of the body.

Static equilibria of convex bodies correspond to the singularities of the
gradient vector field characterizing their surface.
Modern, global theory of the generic singularities of smooth manifolds appears to start with
the papers of Cayley \cite{Cayley} and Maxwell \cite{Maxwell} who independently
introduced topological ideas into this field yielding results on
the global number of stationary points. These ideas have been further generalized
by Poincar\'e and Hopf, leading to the Poincar\'e-Hopf Theorem \cite{Arnold} on topological invariants.
If applied to generic, convex bodies, represented by gradient fields defined on the sphere, this theorem 
states that the number $S$ of `sinks' (stable equilibria), the number $U$ of `sources' (unstable
equilibria) and the number $N$ of saddles always satisfy the equation
\begin{equation} \label{Poincare}
S+U-N=2.
\end{equation}
This formula, the so-called Poincar\'e-Hopf formula can be regarded as a generalization
of the well-known Euler's formula \cite{Euler}  for convex polyhedra.

Static equilibria of polyhedra have also been investigated;
in particular, the mininal number of equilibria attracted substantial interest.
Monostatic polyhedra (i.e. polyhedra with just $S=1$ stable equilibrium point) have been studied
in \cite{Heppes}, \cite{Conway},\cite{Dawson} and \cite{DawsonFinbow}.

The total number $T$ of equlibria $(T= S+U+N)$ has also been in the focus of research.
In planar, homogeneous, convex bodies (rolling along their circumference
on a horizontal support), we have $T\geq 4$ \cite{Domokos1}.
However, convex homogeneous objects with $T=2$ exist in the three-dimensional space (cf. \cite{VarkonyiDomokos}).  Zamfirescu \cite{Zamfirescu} showed that
for \emph{typical} convex bodies, $T$ is infinite. While typical convex bodies
are neither smooth objects, nor are they polyhedral surfaces, Zamfirescu's result
strongly suggests that equilibria in abundant numbers may occur in physically relevant scenarios.

This is indeed the case if we study the surfaces of natural pebbles  which,
while rolling on a horizontal plane, are supported on their convex hull \cite{Domokosetal}.
Their convex hull is well approximated by a many-faceted polyhedron $P$, on which, by studying its detailed 3D scanned images 
one can observe large numbers of adjacent equilibria in strongly
localized \emph{flocks}. If we approximate the polyhedron $P$  by a sufficiently smooth
surface $M$, then we can see that the flocks of equilibria on $P$ appear in the close
vicinity of the (isolated) equilibrium points of $M$ (cf. Figure \ref{fig:pebble}).

In this paper we seek a mathematical justification for this observation.
We study the inverse phenomenon: namely, we seek the numbers and types of static equilibrium points of the families of polyhedra $P^n$ arising as
equidistant $n$-discretizations on an increasingly refined $[u,v]$-grid  of a smooth surface $M$ with $T$ generic equilibrium points, denoted by $m_i$ $(i=1,2,\dots T$).
As $n\to \infty$, $P^n\to M$ and we find that the diameter of each of the $T$ flocks on $P^n$ (appearing around $m_i$) shrink and approach zero.
However, we also find that inside a fixed, $2k \times 2k$ rectangular grid domain (centered at $m_i$), the numbers $S_i^n,U_i^n,N_i^n$ of equilibria
in each flock fluctuate around specific values  $S_i^\star$,$U_i^\star$ and $N_i^\star$ that are independent of the mesh size
and the parametrization of the surface.
We call these quantities the \emph{imaginary equilibrium indices} associated with $m_i$.
We may eliminate the fluctuation of  $S_i^\star$,$U_i^\star$,$N_i^\star$ by averaging over meshes in random
positions (with uniform distributions) and in Theorem~\ref{thm:3Dmainapprox} of Section~\ref{sec:3D}, we obtain the following simple formulae in terms
of the principal curvatures $\kappa_{1,i},\kappa_{2,i} \leq 0$ and the distance $\rho_i$ of $m_i$ from the center of gravity:
\begin{equation} \label{mainresult}
S_i^{\star}=d_i, \hspace{1cm} U_i^{\star}=\kappa_{1,i}\kappa_{2,i}\rho_i^2d_i, \hspace{1cm}  N_i^{\star}=-(\kappa_{1,i}+\kappa_{2,i})\rho_id_i
\end{equation}
where $d_i=1/|(\kappa_{1,i}\rho_i+1)(\kappa_{2,i}\rho_i+1)|$. We also remark that due to (\ref{mainresult}) we have
\[
S_i^\star+U_i^\star-N_i^\star = \left\{ \begin{array}{l} -1, \hbox{ if } m_i \hbox{ is a saddle point,}\\ 1, \hbox{ otherwise,} \end{array}\right.
\]
and thus, by summation over all equilibrium points $m_i$ of $M$, for the full polyhedron $P^n$, the Poincar\'e-Hopf formula is satisfied.

Since imaginary equilibrium indices are defined via local quantities associated with $m_i$, we may interpret them as describing
properties of infinitely fine discretizations. Nevertheless, they also provide approximate prediction for the number and type of localized equilibrium
points for finite, dense meshes.
Beyond those numbers, we also identify the spatial patterns associated with equilibria on fine discretizations
and illustrate these patterns on an ellipsoid (cf. Figure \ref{fig:flocks}).
We also note that these numbers match well with observations on scanned pebble surfaces, as we point out in Section \ref{sec:remarks}, Remark \ref{pebble} and illustrate this
in Figure \ref{fig:pebble}.

In Section~\ref{sec:preliminaries}, we introduce the notations and concepts used in our investigation.
We examine the properties of the equilibrium points of curves in $\Re^2$ in Section~\ref{sec:2D}.
Our results about those of a surface in $\Re^3$ are presented in Section~\ref{sec:3D}.
In Section~\ref{sec:average} we describe a method which, under a slightly special condition, finds a more direct connection between imaginary equilibrium indices
and the numbers of equilibrium points of the approximating surface than the one in the previous sections.
Finally, in Section~\ref{sec:remarks}, we collect our remarks, and propose some open problems.

\section{Preliminaries}\label{sec:preliminaries}

In this paper, we identify points with their position vectors.
For points $p,q \in \Re^d$, we denote the closed segment with endpoints $p, q$ by $[p,q]$.
Furthermore, we denote the standard inner product of $p$ and $q$ by $\langle p,q \rangle$,
and the Euclidean norm of $p$ by $||p||=\sqrt{\langle p,p \rangle}$.
The convex hull of a set $S$ is denoted by $\conv S$, or, if $S=\{ p_1, p_2, \ldots, p_k \}$ is finite, then possibly
by $[p_1,p_2, \ldots, p_k]$.
We denote the origin by $o$.

Since we intend to examine equilibrium points of nonconvex and/or nonsmooth surfaces, we introduce the following types of equilibrium points, which we examine for $d=2$ and $d=3$.

\begin{defn}\label{defn:equi_smooth}
Let $H \subset \Re^d$ be a $C^1$-class hypersurface (i.e. a compact, $1$-co\-di\-men\-sio\-nal $C^1$-submanifold in $\Re^d$), and let $p \in \Re^d \setminus H$. 
If the (unique) tangent hyperplane of $H$ at $m \in H$ is orthogonal to $m-p$, we say that $m$ is an \emph{equilibrium point} of $H$ relative to $p$.
\end{defn}

\begin{defn}\label{defn:equi_poly}
Let $P \subset \Re^d$  be a polyhedral hypersurface (i.e. a compact, $1$-co\-di\-men\-sio\-nal $C^0$-submanifold in $\Re^d$ which is also the polyhedron of a polyhedral complex),
and let $p \in \Re^d \setminus P$.
If, $m \in P$ has a neighborhood $V$ such that the hyperplane $H$, orthogonal to $m-p$ and passing through $m$, supports $\conv (\{ p \} \cup (P \cap V))$,
we say that $m$ is an \emph{equilibrium point} of $P$ relative to $p$.
\end{defn}

For simplicity, we assume that our reference point is the origin $o$.
We deal only with \emph{generic} equilibrium points of $C^2$-hypersurfaces (cf. \cite{Morse}); that is, we assume that at these points the Jacobian of the Euclidean norm function is regular.
In other words, if $\rho$ denotes the distance of $o$ and the equilibrium point $m$,
we assume that $\kappa \neq - \frac{1}{\rho}$, where $\kappa$ is the (signed) curvature of the curve at $m$
for $d=2$, and any of the two fundamental curvatures of the surface at $m$ if $d=3$.

\section{The equilibria of plane curves}\label{sec:2D}

Throughout this section, we assume that the curve under investigation satisfies the $C^3$ differentiability property, and has exactly one equilibrium point $m$.
Note that as a plane curve is a one-dimensional submanifold of $\Re^2$, there is a neighborhood of $m$ in which
the examined curve is given as a simple $r : [\tau_1,\tau_2] \to \Re^2$ three times continuously differentiable curve.
Since for sufficiently fine discretization, all the equilibrium points of the approximating polygon are located in this neighborhood
of $m$, we may assume that the curve is given in the above form.
For simplicity, we may assume that $m=r(0)=(0,\rho)$ with $0 \in (\tau_1,\tau_2)$ and $\rho > 0$, and that the tangent line at $r(0)$ is horizontal;
or in other words, that $\dot{y}(0) = 0$.
Furthermore, we may assume that $\dot{x}(0) > 0$.
We denote the signed curvature of $r$ at $m$ by $\kappa$.

Let $F_n$ denote the $n$-segment equidistant partition of $[\tau_1,\tau_2]$.
If $\frac{i (\tau_2-\tau_1)}{n}$ is a divison point of $F_n$, then we introduce the notation
$p^n_i = r \left( \frac{i (\tau_2-\tau_1)}{n} \right)$.
Let $P^n$ be the polygonal curve with edges of the form $[p^n_{i-1},p^n_i]$.
Note that then $P^n$ is the approximating polygonal curve defined by the equidistant partition $F_n$,
and, furthermore, the vertices of $P^n$ are labeled in a way that
the indices are not necessarily positive integers, but real numbers that are congruent $\mod 1$.
These numbers can be written in the form $\frac{n \tau_1}{\tau_2 - \tau_1} + k$, where $k=0,1,\ldots, n$.

Let $U^n(K)$ denote the number of the equilibrium points of $P^n$ at the vertices $p_i^n$ of $P^n$ satisfying $|i| \leq K$.
Similarly, let $S^n(K)$ be the number of the equilibrium points of $P^n$ lying in the relative interiors of edges $[p_i^n,p_{i+1}^n]$ satisfying $|i| \leq K$.
Our aim is to determine the values of $U^n(K)$ and $S^n(K)$ for `large' values of $n$ and $K$.
We assume that the equilibrium points of $P^n$ are generic; that is, that for any equilibrium at a vertex $p_i^n$, the vector $p_i^n$
is orthogonal to neither $p_i^n-p_{i-1}^n$, nor to $p_i^n-p_{i+1}^n$.
The main result of this section is the following.

\begin{thm}\label{thm:main}
Let $K+K(\kappa,\rho)$ be sufficiently large. Then there is a value of $n_\varepsilon$ such that for every $n > n_\varepsilon$,
we have
\[
U^n(K) = \left\lfloor \frac{|\rho\kappa|}{|\rho\kappa +1|} \right\rfloor \quad \mathrm{ or } \quad
\left\lfloor \frac{|\rho\kappa|}{|\rho\kappa +1|} \right\rfloor+1,
\]
and
\[
S^n(K) = \left\lfloor \frac{1}{|\rho\kappa +1|} \right\rfloor \quad \mathrm{or} \quad
\left\lfloor \frac{1}{|\rho\kappa +1|} \right\rfloor+1.
\]
\end{thm}

\begin{proof}
We use the notations introduced in the preliminary part of this section, and for simplicity,
instead of $p_i^n$ we write only $p_i$.
Let $\Delta_n = \frac{\tau_2-\tau_1}{n}$, which yields that $p_i = r(i \Delta_n)$.

First, we consider the relative interiors of the edges of $P^n$.
Our key observation is that $[p_i,p_{i+1}]$ contains an equilibrium point if, and only if
the following holds:
\begin{equation}\label{eq:equilibrium2}
\langle p_i, p_{i+1}-p_i \rangle < 0 <
\langle p_{i+1}, p_{i+1}-p_i \rangle.
\end{equation}

To find the equilibrium points, we use the second degree Taylor polynomial of $r(\tau)$
with the Lagrange form of the remainder term; that is, we write the curve as
\begin{equation}\label{eq:2Dtaylor}
r(\tau) = r(0)+\dot{r}(0) \tau +
\frac{1}{2} \ddot{r}(0) \tau^2 +
\frac{1}{6} \dddot{r}(\zeta) \tau^3
\end{equation}
for some $\zeta \in (0,\tau)$.

After substituting (\ref{eq:2Dtaylor}) into the left-hand side expression in (\ref{eq:equilibrium2})
and simplifying, we obtain that
\[
\langle p_i, p_{i+1}- p_i \rangle =
(i \dot{x}^2(0) + \frac{1}{2}\rho \ddot{y}(0) + i \rho \ddot{y}(0))\Delta_n^2 + \Err \Delta_n^3,
\]
where $\Err$ is an error term depending on the first three derivatives of $r(\tau)$,
$\Delta_n$ and $i$.
Note that for any fixed value of $i$, or more generally if we keep $i$ bounded, this error term is bounded.
Thus, for any $K$, if $n$ is sufficiently large and if $|i| \leq K$, then the sign of this expression is
determined by the first three terms.
In other words, in this case the left-hand side inequality of (\ref{eq:equilibrium2}) is satisfied if and only if
\[
i \dot{x}^2(0) + \frac{1}{2}\rho \ddot{y}(0) + i \rho \ddot{y}(0) < 0
\]

Note that $\kappa = \frac{\ddot{y}(0)}{\dot{x}^2(0)}$.
Thus, we have that
\[
i (\rho \kappa + 1) < -\frac{1}{2} \rho \kappa .
\]

By a similar argument we obtain that for any value of $|i| \leq K$, if $n$ is sufficiently large,
then the right-hand side inequality of (\ref{eq:equilibrium2}) is equivalent to
\[
-1-\frac{1}{2} \rho \kappa < i (\rho \kappa + 1).
\]
Thus, for any $|i| \leq K$, $\relint [p_i,p_{i+1}]$ contains an imaginary equilibrium point if and only if
\begin{equation}\label{eq:equilibrium4}
-1-\frac{1}{2} \rho \kappa < i (\rho \kappa + 1) < -\frac{1}{2} \rho \kappa.
\end{equation}
Hence, the number of these edges is no more, and if $K$ is sufficiently large then it is equal to, the different values of $i$ that fit in a certain interval of length $\frac{1}{|\rho \kappa + 1|}$.
Since the possible values of $i$ are congruent $\mod 1$,
this number is either $\left\lfloor \frac{1}{|\rho \kappa + 1|} \right\rfloor$, or one more.

Now we deal with the vertices of $P^n$.
Note that $p_i$ is an equilibrium point if, and only if one of the following holds:
\begin{equation}\label{eq:2dvertex1}
\langle - p_i, p_{i+1}- p_i \rangle \geq 0 \quad \mathrm{and}
\langle - p_i, p_{i-1}- p_i \rangle \geq 0,
\end{equation}
or
\begin{equation}\label{eq:2dvertex2}
\langle - p_i, p_{i+1}- p_i \rangle \leq 0 \quad \mathrm{and}
\langle - p_i, p_{i-1}- p_i \rangle \leq 0.
\end{equation}
Using the same kind of argument that we used for edges, we obtain that for
sufficiently large values of $n$ and $K$ and for $|i| \leq K$,
the systems of equation above are equivalent to
\[
\frac{1}{2}\rho\kappa \leq i (\rho \kappa +1) \leq - \frac{1}{2}\rho \kappa
\]
and
\[
- \frac{1}{2}\rho\kappa \leq i (\rho \kappa +1) \leq \frac{1}{2}\rho \kappa,
\]
respectively.

Clearly, if $\kappa > 0$, then the first system can be satisfied for no value of $i$,
and so is the second one if $\kappa < 0$.
Thus, in both cases, the number of the values of $i$ satisfying the corresponding system of equations
is either $\left\lfloor \frac{\rho | \kappa |}{|\rho \kappa + 1|} \right\rfloor$ or one more.
\end{proof}

\section{The equilibria of surfaces}\label{sec:3D}

Similarly like in the previous section, we may assume that the surface, given in the form $r : D \to \Re^3$
with $(u,v) \in D=[u_1,u_2] \times [v_1,v_2]$, satisfies the $C^3$ property.
For simplicity, we examine only the case that the surface has exactly one equilibrium point, namely $m=r(0,0)=(0,0,\rho)$
with $(0,0) \in \inter D$ and $\rho > 0$.
Without loss of generality, we may assume that the tangent plane of $r$ at $m$ is horizontal, and that
its normal vector $n^\star = r_u(0,0) \times r_v(0,0)$ is an outer normal vector of the surface; that is, it points in the direction of the positive half of the $z$-axis.
We denote the two principal curvatures of the surface at $m$ by $\kappa_1$ and $\kappa_2$, and note that, dealing with generic equilibria, we assume
in the rest of the section that $\rho \kappa_1 + 1 \neq 0 \neq \rho \kappa_2 + 1$.
To be able to define the approximating surface, we assume that $r(D)$ is convex; more specifically,
that $r(D)$ is a subset of the boundary of a compact, convex set with nonempty interior, containing $o$ in its interior, which yields, in particular,
that $\kappa_1, \kappa_2 \leq 0$.

Consider an equidistant partition of the rectangle $D$ into $n^2$ homothetic copies of ratio $\frac{1}{n}$.
Similarly like in the previous section, for the indices of the division points of the partition
we use not necessarily integers, but real numbers that are congruent $\mod 1$.
More precisely, we set $p^n_{i,j} = r\left( i\frac{u_2-u_1}{n},j \frac{v_2-v_1}{n} \right)$.
For simplicity, for the point $p^n_{i,j}$, we may use the notation $p_{i,j}$.
Furthermore, we set $\Delta_{u,n}= \frac{u_2-u_1}{n}$, $\Delta_{v,n} = \frac{v_2-v_1}{n}$ and $\lambda = \frac{\Delta_{v,n}}{\Delta_{u,n}}$.

Let $S = \{ p_{i,j}, p_{i+1,j}, p_{i+1,j+1}, p_{i,j+1} \}$, and set $Q=\conv ( S \cup \{ o \})$.
Since for sufficiently large values of $n$ the equilibrium points of $P^n$ are contained in a small neighborhood of $m$ and since $r_u(u,v)$ and $r_v(u,v)$ are continuous functions,
we may assume that $o$ is a vertex of this polyhedron,
and that the rays emanating from $o$ and passing through $p_{i,j}, p_{i+1,j}, p_{i+1,j+1}, p_{i,j+1}$, respectively, are in counterclockwise order.

If $S$ is not coplanar, then either $[p_{i,j},p_{i+1,j+1}]$ or $[p_{i+1,j},p_{i,j+1}]$ is an edge of $Q$, depending on the position of the points of $S$.
Thus, in this case $Q$ has two triangle faces not containing $o$, namely either $[p_{i,j},p_{i+1,j},p_{i+1,j+1}]$ and $[p_{i,j},p_{i+1,j+1},p_{i,j+1}]$, or
$[p_{i,j},p_{i+1,j},p_{i,j+1}]$ and $[p_{i+1,j},p_{i+1,j+1},p_{i,j+1}]$. We denote these faces $F^1_{i,j}$ and $F^2_{i,j}$.
If $S$ is coplanar, we dissect it into two triangles by either $[p_{i,j},p_{i+1,j+1}]$ or $[p_{i+1,j},p_{i,j+1}]$, and regard it as the union of two triangle faces.
Now we define the approximating polyhedral surface $P^n$ as the union of the faces $F^1_{i,j}$ and $F^2_{i,j}$ for all possible indices $(i,j)$.
The vertices of $P^n$ are the points $p_{i,j}$, and the edges and the faces are those described in this paragraph.
We observe that $P^n$ is simplicial, and that it is not necessarily a subset of the boundary of a convex polyhedron.

The equilibrium points of $P^n$ can be vertices, or relative interior points of edges or faces of $P^n$.
We assume that each of these equilibria is generic, that is, if $q$ is an equilibrium of $P^n$ in the relative interior of a face $F$ of dimension $0$, $1$ or $2$,
then for some neighborhood $V$ of $q$, the plane passing through and orthogonal to $q$ contains no point of $V$ other than those of $F$.
We denote the number of the equilibrium points of $P^n$ at a vertex $p_{i,j}$ satisfying $|i| \leq K$ and $|j| \leq K$, by $U^n(K)$.
Similarly, the numbers of the equilibrium points on an edge or a face of $P^n$,
which has a vertex $p_{i,j}$ satisfying $|i|\leq K$ and $|j| \leq K$, by $N^n(K)$ and $S^n(K)$, respectively.
Our aim is to estimate $U^n(K)$, $N^n(K)$ and $S^n(K)$ if $n$ and $K$ are sufficiently large.

Note that, by the conditions for $r$ described before, we have $r_u(0,0) = ( x_u, y_u, 0 )$ and $r_v(0,0) = (x_v, y_v, 0)$,
and that the outer normal vector of the surface at $m$ is
\[
n^{\star} = \frac{r_u (0,0) \times r_v(0,0)}{||r_u (0,0) \times r_v(0,0)||} = (0, 0, 1)
\]

We recall the notions of the first and the second fundamental quantities of a surface. In our setting, for these at the point $m$ we have
\[
E = x_u^2+y_u^2, F= x_u x_v + y_uy_v \hbox{ and } G=x_v^2 + y_v^2.
\]
Furthermore, $L$, $M$ and $N$ are the $z$-coordinates of $r_{uu}(0,0)$, $r_{uv}(0,0)$
and $r_{vv}(0,0)$, respectively.
It is well-known in differential geometry that
\begin{equation}\label{eq:areavector}
EG-F^2 = \big(r_u(0,0) \times r_v(0,0)\big)^2 = (x_uy_v-x_vy_u)^2,
\end{equation}
and that
\begin{equation}\label{eq:princcurve}
\kappa_1 \kappa_2 = \frac{LN-M^2}{EG-F^2}, \hbox{ and } \kappa_1 + \kappa_2 = \frac{EN-2FM+GL}{EG-F^2}.
\end{equation}
According to our conditions, we have $E, G > 0$, $EG-F^2 > 0$, $L,N \leq 0$ and $LN-M^2 \geq 0$.

In our investigation, we examine only the case that $\lambda M \leq L$ and $\lambda M \leq \lambda^2 N$.
Since $LN-M^2 \geq 0$, these two inequalities are satisfied for some values of $\lambda$; for instance,
for the value satisfying $L = \lambda^2 N$.
These conditions ensure that the discretization mesh sizes do not differ radically in the $u$ and $v$ directions.
We note that, in order to obtain the area of a surface by taking the limit of an approximating triangulated surface,
a similar condition is necessary.

In the formulation of the main result of this section, we need the following notations.

\noindent
$\begin{array}{ll}
V_1 = \lambda \rho L(G+\rho N), & W_1 = \lambda^2 \rho N(F+\rho M),\\
V_2 = \rho L(F+\rho M), & W_2 = \lambda \rho N(E+\rho L),\\
V_3 = \lambda^2 \rho(NF-MG), & W_3 = \lambda((EG-F^2)-\rho(MF-LG)),\\
V_4 = \lambda((EG-F^2)-\rho(MF-NE)), & W_4 = \rho(LF-ME),\\
V_5 = \lambda(\rho(NE-MF)+(EG-F^2)), & W_5 = \lambda \rho^2 (\lambda |N| - |M|)(ME-LF),\\
V_6 = F+\rho M, & W_6 = \rho (\lambda |N| - |M|)(E+\rho L),\\
V_7 = \lambda^2 \rho (NF-MG), & W_7 = \lambda \rho (|L| - \lambda |M|)(\rho (LG-MF)+(EG-F^2)),\\
V_8 = \lambda(G+\rho N), & W_8 = \rho (|L| - \lambda |M|)(F+\rho M),\\
V_9 = V_5, & W_9 = \lambda \rho^2 |M|(ME-LF),\\
V_{10} = V_6, & W_{10} = \rho |M|(E+\rho L),\\
V_{11} = V_7, & W_{11} = \lambda^2 |M|(\rho (LG-MF)+(EG-F^2)),\\
V_{12} = V_8, & W_{12} = \rho \lambda |M|(F+\rho M),\\
\end{array}$
and
\begin{eqnarray*}
\Err_U & = & 2+ \frac{\sum\limits_{s=1}^2 \max\{|V_s|,|W_s|,|V_s-W_s|, |V_s+W_s|\} }{\lambda(EG-F^2)|(\kappa_1 \rho + 1)(\kappa_2 \rho + 1) |},\\
\Err_S & =& 4+ 2 \frac{ \sum\limits_{s=3}^4 \max\{|V_s|,|W_s|,|V_s-W_s|, |V_s+W_s|\} }{\lambda (EG-F^2)
|(\kappa_1 \rho + 1) (\kappa_2 \rho + 1)|},\\
\Err_N & = & 6 + \frac{\sum\limits_{s=5}^{12}\max\{|V_s|, |W_s|, |V_s-W_s|, |V_s + W_s|\} }{\lambda (EG-F^2)
|(\rho \kappa_1 + 1)(\rho \kappa_2 + 1)|}.
\end{eqnarray*}

Our main result is the following.

\begin{thm}\label{thm:3Dmainapprox}
Let $K=(\rho,\kappa_1,\kappa_2,\lambda)$ be sufficiently large. Then there is a value $n_\varepsilon$ such that for every $n> n_\varepsilon$,
we have
\begin{eqnarray*}
\left| U^n(K) - \frac{\rho^2 \kappa_1 \kappa_2}{|(\kappa_1 \rho + 1)(\kappa_2 \rho + 1) |} \right| & \leq & \Err_U,\\
\left| S^n(K) - \frac{1}{|(\kappa_1 \rho + 1)(\kappa_2 \rho + 1) |} \right| & \leq & \Err_S,\\
\left| N^n(K) + \frac{(\kappa_1 + \kappa_2)\rho}{|(\kappa_1 \rho + 1)(\kappa_2 \rho + 1) |} \right| & \leq & \Err_N.
\end{eqnarray*}
\end{thm}

We can interpret this theorem in the following way.
Our result in Theorem~\ref{thm:main} can be interpreted in a way that the number of a certain type of equilibrium point fluctuates around a specific number that depend only of local 
quantities of the curve at $m$, within an error one. In general, this fluctuation seems random as we compute this numbers for different values of $n$.
Here we have a similar result, with the only difference that the error term is not one, but a complicated expression of the fundamental quantities of the surface and $\lambda$.
In general, this fluctuation seems random; that is, if we compute these numbers for different values of $n$, then the discrepancy between the measured number and the corresponding estimate in the theorem 'averages out'. For more information about the sizes of the error terms, the reader is referred to Remark~\ref{rem:errorterms}, and for a more precise treatment of the fluctuation
to Section~\ref{sec:average}. The error functions in Theorem~\ref{thm:3Dmainapprox} provide a global bound and may in many 
cases grossly overestimate the actual errors. In particular, the mesh ratio 
$\lambda$ as well as the topology of flocks may radically influence the 
actual errors. This can also be observed in our example illustrated in 
 Figure~\ref{fig:flocks} showing a triaxial ellipsoid with axis ratios a:b:c=1.23:1.15:1.

We start with two lemmas that we need in the proof.

\begin{lem}\label{lem:expectation}
Let $S \subset \Re^d$ a nonempty Lebesgue measurable set with finite measure $\lambda(S)$.
Let $p_k$ denote the probability that $x+S$ contains exactly $k$ points of the lattice $\Z^d$, where $x$ is chosen from the cube $[0,1]^d$ using
the uniform distribution.
Define the probability distribution $\sigma$ by $P(\sigma = k) = p_k$.
Then the expected value of $\sigma$ is $E(\sigma) = \lambda(S)$.
\end{lem}

\begin{proof}
First, note that as $\lambda(S)$ is finite, we may write $S$ as the disjoint union
of finitely many Lebesgue measurable sets with diameters less than one. Let these sets be
$S_1, S_2, \ldots, S_m$. Let $\sigma_i$ be the distribution such that
$P(\sigma_i = k)$ is the probability that $x+S_i$ contains $k$ lattice points, where $x$ is chosen
uniformly from the unit cube of $\Re^d$.
By the linearity of the expectation, we have
\[
E(\sigma) = \sum\limits_{i=1}^m E(\sigma_i).
\]

We show that $E(\sigma_i) = \lambda(S_i)$, from which the assertion readily follows.
Without loss of generality, we may assume that $S \subset (-1,0)^d$.
Then for any $x \in [0,1]^d$, $x+S_i$ either contains the origin, or it does not contain any point of $\Z^d$.
Note that the set of points $x$ for which $x+S_i$ contains the origin is $-S_i$, and thus, $E(\sigma_i)=\lambda(S_i)$.
\end{proof}

We remark that, clearly, the lattice $\Z^d$ in the formulation of Lemma~\ref{lem:expectation}
can be replaced by any lattice, if $x$ is chosen from the corresponding fundamental lattice parallelotope.

\begin{lem}\label{lem:3Dwhichedge}
For every value of $K$, if $n$ is sufficiently large then for every vertex $p_{i,j}$
of $P^n$ with $|i| \leq K$ $|j| \leq K$ we have the following:
\begin{itemize}
\item If $M>0$, then $E_+=[p_{i,j},p_{i+1,j+1}]$ is an edge of $P^n$, and
\item if $M<0$, then $E_-=[p_{i+1,j},p_{i,j+1}]$ is an edge of $P^n$.
\end{itemize}
\end{lem}

\begin{proof}
Recall that if $p_{i,j}$ is sufficiently close to $m$, then
$o$ and the endpoints of $E_+$ and $E_-$ are in convex position, and the rays emanating from $o$
and passing through $p_{i,j}, p_{i+1,j}, p_{i+1,j+1}$ and $p_{i,j+1}$ are
in this counterclockwise order in their conic hull.

Let $S_+$ denote the sum of the volumes of the tetrahedra $[o,p_{i,j},p_{i+1,j},p_{i+1,j+1}]$ and $[o,p_{i+1,j+1},p_{i,j+1},p_{i,j}]$, and $S_-$ the sum of the volumes of $[o,p_{i,j},p_{i+1,j},p_{i,j+1}]$
and $[o,p_{i+1,j},p_{i+1,j+1},p_{i,j+1}]$.
Observe that if $S_+ - S_- > 0$, then $E_+$ is an edge of $P^n$, if $S_+ - S_- < 0$ then $E_-$,
and if $S_+ - S_- = 0$, then the two segments are coplanar.
Note that the signed volume of the tetrahedron spanned by the vectors $a, b$ and $c$ is
$\frac{1}{6}\langle a, b \times c \rangle$.
Thus, the sign of $S_+ - S_-$ coincides with the sign of
\begin{eqnarray*}
V & = &\langle p_{i,j},p_{i+1,j} \times p_{i+1,j+1} \rangle+ \langle p_{i+1,j+1},p_{i,j+1} \times p_{i,j} \rangle \\
& & - \langle p_{i,j},p_{i+1,j} \times p_{i,j+1} \rangle - \langle p_{i+1,j},p_{i+1,j+1} \times p_{i,j+1} \rangle .
\end{eqnarray*}

We use the second degree Taylor polynomial of $r(u,v)$ to approximate the surface.
In other words, we write the surface in the form
\begin{equation}\label{eq:3dTaylor}
r(u,v)= r(0,0) + r_u (0,0) u + r_v(0,0) v + \frac{1}{2} r_{uu}(0,0) u^2 +
r_{uv}(0,0) uv +
\end{equation}
\[
+ \frac{1}{2} r_{vv}(0,0) v^2 + K_1 u^3 + K_2 u^2 v + K_3 u v^2 + K_4 v^3,
\]
where $K_1, K_2, K_3$ and $K_4$ are bounded $\Re^2 \to \Re^3$ functions of $u$ and $v$.

After computing $V$ and substituting $\Delta_{v,n} = \lambda \Delta_{u,n}$,
we obtain a polynomial of $\Delta_{u,n}$.
In this polynomial, the first nonzero coefficient, which is that of $\Delta_{u,n}^4$,
is $M \sqrt{EG-F^2}$.
Thus, for any $|i| \leq K$ and $|j| \leq K$, if $n$ is sufficiently large, then $\sign (S_+ - S_-) = \sign M$.
Hence, if $M > 0$ then $E_+$ is an edge of the convex hull,
if $M < 0$ then $E_-$, and if $M=0$ then it can be both:
it is determined by the higher degree terms in the approximation.
\end{proof}

Now we prove Theorem~\ref{thm:3Dmainapprox}.

\begin{proof}[Proof of Theorem~\ref{thm:3Dmainapprox}]
We prove the three formulas separately.

\noindent
\emph{Part 1: The proof of the assertion for $U^n(K)$.}

As in the proof of Lemma~\ref{lem:3Dwhichedge}, we approximate the surface $r(u,v)$
by its second degree Taylor polynomial.
We start with the observation that $p_{i,j}$ is an equilibrium point if, and only if,
each of the eight angles of the form $\angle (op_{i,j}p_{i+\delta_i,j+\delta_j})$,
where $\delta_i, \delta_j \in \{ -1, 0, 1\}$ and $(\delta_i,\delta_j) \neq (0,0)$, is acute, or each is obtuse.
We note that by Lemma~\ref{lem:3Dwhichedge}, it is sufficient to examine six of the angles.

First, we examine the condition under which $\angle (op_{i,j}p_{i+1,j}) < \frac{\pi}{2}$.
Note that this condition is equivalent to
\begin{equation}\label{eq:csucs1bev}
\langle -p_{i,j}, p_{i+1,j}-p_{i,j} \rangle > 0.
\end{equation}
After substituting (\ref{eq:3dTaylor}) and $\Delta_{v,n} = \lambda \Delta_{u,n}$ into this inequality
and then expanding, the first nonzero term, which does not depend on $K_1, K_2, K_3$ and $K_4$, is of degree two.
Since the coefficients of the remaining terms are bounded functions for any $|i| \leq K$ and $|j| \leq K$,
if $n$ is sufficiently large, then (\ref{eq:csucs1bev}) is satisfied
if and only if the first nonzero term is positive.
We can rewrite this term using the fundamental quantities of the surface, which yields the following inequality:
\begin{equation}\label{eq:csucs1}
X < - \frac{1}{2} \rho L,
\end{equation}
where $X = i (E+\rho L) + \lambda j (F+\rho M)$.
From the condition $\angle (op_{i,j}p_{i-1,j})  < \frac{\pi}{2}$, we obtain similarly that
\begin{equation}\label{eq:csucs2}
\frac{1}{2} \rho L < X.
\end{equation}

We remark that from $\angle (op_{i,j}p_{i+1,j}) > \frac{\pi}{2}$ and $\angle (op_{i,j}p_{i-1,j})  > \frac{\pi}{2}$
it follows similarly that if $n$ is sufficiently large, then $ - \frac{1}{2} \rho L < X < \frac{1}{2} \rho L$.
Since $L \leq 0$, it is a contradiction, and thus, we have that if $p_{i,j}$ is an equilibrium point,
then all the eight angles are acute.

From the condition that the remaining six angles are acute, we have that
\begin{equation}\label{eq:csucs3-4}
\frac{1}{2} \lambda^2 \rho N < Y < - \frac{1}{2} \lambda^2 \rho N,
\end{equation}
where $Y = \lambda i (F+\rho M) + \lambda^2 j (G+\rho N)$, and
\begin{equation}\label{eq:csucs5-8}
\begin{array}{ccccc}
 \frac{1}{2} \lambda^2\rho   N + \lambda \rho  M + \frac{1}{2} \rho  L & < & X + Y & < & -\frac{1}{2}
 \lambda^2 \rho  N - \lambda \rho  M - \frac{1}{2} \rho  L\\
\frac{1}{2} \lambda^2 \rho N - \lambda \rho  M + \frac{1}{2} \rho  L & < & Y - X & < & -\frac{1}{2}
\lambda^2 \rho  N + \lambda \rho  M - \frac{1}{2} \rho  L
\end{array}
\end{equation}

Let $P$ denote the convex polygonal region in the $(X,Y)$-plane defined by the inequalities in (\ref{eq:csucs1}),
(\ref{eq:csucs2}), (\ref{eq:csucs3-4}) and (\ref{eq:csucs5-8}).
We determine the shape of $P$.

Observe first that as $L, N \leq 0$, the first four inequalities define a (possibly degenerate) rectangle $R$.
Next, we note that the two pairs of inequalities in (\ref{eq:csucs5-8})
define two (possibly degenerate) infinite strips.
Recall that $LN-M^2 \geq 0$, which implies that the quadratic forms of $\lambda$ on the right-hand sides of
(\ref{eq:csucs5-8}) are positive semidefinite.
Thus, the two infinite strips defined in (\ref{eq:csucs5-8}) are of nonnegative widths.
From (\ref{eq:csucs5-8}), it follows also that if $M \leq 0$,
then the first strip contains $R$, and otherwise the second one.
Furthermore, since $\lambda |M| \leq |L|, \lambda^2 |N|$, the other strip contains exactly two
(opposite) vertices of $R$. 
Hence, $P$ is a (possibly degenerate) centrally symmetric hexagon
that is obtained as the rectangle $R$ truncated with parallel lines at two opposite vertices.
This hexagon can be observed in Part (c) of Figure~\ref{fig:flocks}

We have that, for sufficiently large $n$, the number of the equilibrium points at the vertices of $P^n$
is equal to the points of a translate of the lattice $\{ (X(i,j),Y(i,j)): i,j \in \Z^2 \}$ contained in $P$;
or, equivalently, the number of lattice points contained in a translate of $P$.
Let $Q$ denote the fundamental parallelogram of the lattice; that is,
$Q= \{ (X(i,j), Y(i,j)) : i,j \in [0,1] \}$.
Then, by Lemma~\ref{lem:expectation}, choosing the translate of $P$ `randomly', the expected value of 
the number of lattice points in it is $\frac{\area(P)}{\area(Q)}$.

It is an elementary computation to show that $\area(P) = \lambda^2 \rho^2 (LN-M^2)$.
On the other hand, $\area(Q)$ is the absolute value of the determinant $D$
of the linear transformation $T$ defining $X$ and $Y$.
Expanding $D$ and using the identities in (\ref{eq:princcurve}), we obtain that
\[
\frac{\area(P)}{\area(Q)} = \frac{\rho^2 \kappa_1 \kappa_2}{\left| (\rho \kappa_1 + 1)(\rho \kappa_2+1)\right|}.
\]

To prove the assertion it suffices to show that the discrepancy between this ratio and the number of the equilibrium points is within $\Err_U$.
Because of the properties of the expected value, for this we need only show that
the difference between the numbers of the lattice points contained in two different translates of $P$ is at most
$\Err_U$, or in other words, that the difference between the numbers of the points of $\Z^2$
contained in two different translates of $T^{-1} (P)$ is at most $\Err_U$.

Note that if we move a convex region parallel to the $x$-axis, then on each line parallel to the $x$-axis, the number
of the points of $\Z^2$ can change by at most one. Thus, the discrepancy is at most the vertical width of the region, rounded up.
We can observe the same property if we move a convex region parallel to the $y$-axis.
Thus, the error term is at most two more than the half-perimeter of the smallest axis-parallel rectangle containing $T^{-1} (P)$.
Clearly, we can estimate this quantity from above by computing the corresponding quantity for the parallelogram $T^{-1}(R)$ instead of the hexagon $T^{-1} (P)$.
Hence, the error term is $Err_U=2+(x_{\max}-x_{\min})+(y_{\max}-y_{\min})$, where $x_{\max}$, $x_{\min}$, $y_{\max}$
and $y_{\min}$ are the extrema of the coordinates of the points of $T^{-1}(R)$.
Clearly, these extrema are obtained at vertices of $T^{-1}(R)$.
We leave to the reader to verify that by computing these vertices we obtain the error term $\Err_U$
of Theorem~\ref{thm:3Dmainapprox}.

\noindent
\emph{Part 2: The proof of the assertion for $S^n(K)$.}

We prove the statement in three different cases.

\emph{Case 1,} $M > 0$.
In this case we need to determine the conditions under which there are equilibrium points on
$[p_{i,j},p_{i+1,j},p_{i+1,j+1}]$ or $[p_{i+1,j+1},p_{i,j+1},p_{i,j}]$.

Consider the face $[p_{i,j},p_{i+1,j},p_{i+1,j+1}]$.
First, we note that an outer normal vector of this triangle face is
\[
\hat{n} = (p_{i+1,j}-p_{i,j}) \times (p_{i+1,j+1} - p_{i,j})= p_{i,j} \times p_{i+1,j} + p_{i+1,j} \times
p_{i+1,j+1} + p_{i+1,j+1} \times p_{i,j}.
\]
Our key observation is the following: There is an equilibrium point on the triangle if and only if
the ray emanating from the origin and perpendicular to the plane of the triangle intersects the triangle.
In other words, it happens if and only if $\hat{n}$ is contained in the conic hull of the triangle;
that is, if the angles between $\hat{n}$ and the inner normal vectors of the faces of the tetrahedron
$[o,p_{i,j},p_{i+1,j},p_{i+1,j+1}]$, apart from $-\hat{n}$, are acute.
Thus, we have three conditions that determine if there is an equilibrium point on the face:
\[
\langle \hat{n} , p_{i,j} \times p_{i+1,j} \rangle > 0;
\]
\[
\langle \hat{n} , p_{i+1,j} \times p_{i+1,j+1} \rangle > 0;
\]
\[
\langle \hat{n} , p_{i+1,j+1} \times p_{i,j} \rangle > 0.
\]

After substituting (\ref{eq:3dTaylor}) and $\Delta_{v,n} = \lambda \Delta_{u,n}$ into these inequalities and expanding,
the first nonzero terms, which do not depend on $K_1, K_2, K_3$ and $K_4$, are of degree $4$ in $\Delta_{u,n}$.
Expressing them with the fundamental quantities of the surface, we obtain that
\begin{equation}\label{eq:face1}
\lambda \rho ME < X, \quad \lambda^2 \rho MF  - \lambda^2(EG-F^2) < Y, \quad X+Y < \lambda \rho ME + \lambda^2 \rho MF,
\end{equation}
where
\[
X = \lambda i \rho (LF-ME) + \lambda^2 j (\rho(MF-NE)-(EG-F^2)) +
\frac{1}{2}\lambda \rho LF-\frac{1}{2}\lambda^2 \rho NE
\]
and
\[
Y = \lambda^2 i (\rho(LG-MF)+(EG-F^2)) + \lambda^3 j \rho (MG-NF) + \frac{1}{2}\lambda^2 \rho LG -
\frac{1}{2} \lambda^3 \rho NF.
\]

Similarly like in Part 1, we estimate the number of the faces with an equilibrium point
with the ratio of the area of the triangle defined in (\ref{eq:face1}) to
the absolute value of the determinant of the affine transformation defining $X$ and $Y$.
An elementary computation yields that the determinant $D$ is
\[
D = \lambda^4 (EG-F^2)^2 (\kappa_1 \rho +1)(\kappa_2 \rho + 1)
\]
and, furthermore, that the inequalities in (\ref{eq:face1}) define a right triangle $T_1$ the legs of which
are of length $\lambda^2 (EG-F^2)$.
Thus, the ratio of the area of $T_1$ to the area of the fundamental parallelogram of the affine transformation
is:
\[
\frac{1}{2\left| (\kappa_1 \rho +1)(\kappa_2 \rho + 1) \right|}.
\]

Now we consider the face $[p_{i+1,j+1},p_{i,j+1},p_{i,j}]$.
An argument similar to the previous one yields the inequalities
\begin{equation}\label{eq:face2}
X < \lambda^2(EG-F^2) - \lambda^2 \rho MF, \quad Y < - \lambda^3 \rho MG, \quad -\lambda^2 \rho MF -
\lambda^3 \rho MG < X + Y,
\end{equation}
where $X$ and $Y$ are defined as in the previous case.
This region is a reflected copy $T_2$ of the triangle $T_1$, and thus
for the ratio of its area and the absolute value of the determinant we obtain the same quantity.

Note that as $M > 0$ and as $E + \lambda F + \lambda^2 G$ is a positive definite quadratic form,
the open half planes defined with the third inequalities of (\ref{eq:face1}) and (\ref{eq:face2})
overlap.
Nevertheless, it may happen that $T_1$ and $T_2$ do not overlap; it happens, for instance,
when $\rho$ is sufficiently large.
These triangles can be observed in Parts (d0), (d1) and (d2) of Figure~\ref{fig:flocks}.

\emph{Case 2}, $M < 0$.
We may apply a consideration similar to that in Case 1.
The faces we need to examine are $[p_{i,j},p_{i+1,j},p_{i,j+1}]$ and $[p_{i+1,j},p_{i+1,j+1},p_{i,j+1}]$.
For the first face, we have the inequalities
\begin{equation}\label{eq:face3}
0 < X, \quad Y < 0, \quad - \lambda^2 (EG-F^2)< Y-X,
\end{equation}
and for the second one, we have
\[
X < \lambda^2 (EG-F^2) + \lambda \rho ME - \lambda^2 \rho MF,
\]
\begin{equation}\label{eq:face4}
\lambda^2 \rho MF - \lambda^3 \rho MG - \lambda^2 (EG-F^2) < Y,
\end{equation}
\[
Y-X < - \lambda \rho ME + 2 \lambda^2 \rho MF - \lambda^3 \rho MG - \lambda^2(EG-F^2),
\]
where $X$ and $Y$ are defined as in Case 1 of Part 2.
These two regions define two triangles. By computing their area
we obtain the same estimate for the number of the equilibrium points as in Case 1.

\emph{Case 3}, $M = 0$.
In this case the coefficient of the the term of degree four in the polynomial $V$ in Lemma~\ref{lem:3Dwhichedge} does not
determine which of the edges $E_+=[p_{i,j},p_{i+1,j+1}]$ and $E_-=[p_{i+1,j},p_{i,j+1}]$ belongs to $P^n$.
Note that to estimate $S^n(K)$ we may use the triangles defined by the inequalities in (\ref{eq:face1}) and (\ref{eq:face2}) from Case 1 if this edge is $E_+$,
and the triangles defined by (\ref{eq:face3}) and (\ref{eq:face4}) if it is $E_-$.
Note that in both cases, these two triangles degenerate into the same square in the $(X,Y)$-plane if $M=0$
(and in this case the two triangles do not overlap).
This square is defined by
\begin{equation}\label{eq:rectangle}
0 < X < \lambda^2 (EG-F^2), - \lambda^2 (EG-F^2) < Y < 0.
\end{equation}
Thus, in the case $M=0$, if a triangle with vertices from amongst $p_{i,j}, p_{i+1,j}, p_{i+1,j+1}$ and $p_{i,j+1}$
contains an equilibrium point, then there is exactly one such face of $P^n$ no matter which edge, $E_+$ or $E_-$ belongs to $P^n$.

We have yet to determine the error term $\Err_S$ of the theorem for an arbitrary value of $M$.
Our method is similar to the one used in Part 1.
Let $T$ be the affine transformation defined by the formulas for $X$ and $Y$.
Let $S$ be an axis parallel square of side length $\lambda^2(EG-F^2)$.
Then $\Err_S$ in any of the cases is at most four more than the perimeter of the parallelogram $T^{-1}(S)$.
Hence, a simple computation yields the required quantity.

\noindent
\emph{Part 3: The proof of the assertion for $N^n(K)$.}

The proof is similar to those on Parts 1 and 2, and thus we just sketch it.

\emph{Case 1}, $M > 0$.
We need to find the conditions under which $[p_{i,j},p_{i+1,j}]$, $[p_{i,j},p_{i,j+1}]$
or $[p_{i,j},p_{i+1,j+1}]$ contains an equilibrium point. For simplicity, we call these edges
\emph{horizontal, vertical} and \emph{diagonal}, respectively.

Consider a horizontal edge $[p_{i,j},p_{i+1,j}]$.
Note that it contains an equilibrium point if and only if the following holds:
\begin{itemize}
\item The angles $\angle (o,p_{i,j},p_{i+1,j})$ and $\angle (o,p_{i+1,j},p_{i,j})$ are acute.
\item The two angles between the triangle $[o,p_{i,j},p_{i+1,j}]$, and
$[p_{i,j},p_{i,j-1},p_{i+1,j}]$ and $[p_{i,j},p_{i+1,j},p_{i+1,j+1}]$, are either both acute, or both obtuse.
\end{itemize}

Using the second degree Taylor polynomial, for the first nonzero terms we obtain the
following inequalities:
\begin{equation}\label{eq:edge1}
0 < X_h < E, 0 < Y_h < - \frac{1}{2} \lambda \rho E (\lambda N + M),
\end{equation}
or
\[
0 < X_h < E, - \lambda \rho E (\lambda N + M) < Y_h < 0,
\]
where
\[
X_h = -i(E+\rho L)-\lambda j(F+\rho M)-\frac{1}{2} \rho L,
\]
and
\[
Y_h = \lambda i \rho (ME-LF) + \lambda^2 j \left((EG-F^2) +\rho (NE-MF) \right)
- \lambda \rho (\lambda NE - MF).
\]
Observe that due to our assumption that $\lambda |M| \leq \lambda^2 |N|$,
the second system of inequalities is not satisfied for any value of $Y_h$.
Note that the first system determines a rectangle $R_h$ in the $(X,Y)$ coordinate-system.

We estimate the number of horizontal edges with an equilibrium point as usual, and obtain the quantity
\[
Z_h = \frac{-E \rho (\lambda N + M)}{\lambda (EG-F^2) | (\rho \kappa_1 +1)(\rho \kappa_2 + 1)|}.
\]

For the vertical edge $[p_{i,j},p_{i,j+1}]$, using similar quantities $X_v$ and $Y_v$,
we obtain another rectangle $R_v$ and the estimate
\[
Z_v = \frac{-G \rho (L+\lambda M)}{(EG-F^2) | (\rho \kappa_1 +1)(\rho \kappa_2 + 1)|}.
\]

For the edge $[p_{i,j},p_{i+1,j+1}]$, we have the inequalities
\begin{equation}\label{eq:edge2}
0 < X_d < E +2 \lambda F + \lambda^2 G, 0 < Y_d < \lambda M (E +2\lambda F + \lambda^2 G)
\end{equation}
or
\[
0 < X_d < E +2 \lambda F + \lambda^2 G, \lambda M (E +2\lambda F + \lambda^2 G) < Y_d < 0,
\]
where
\[
X_d = - i \big( E+\rho L+\lambda(F+\rho M) \big)-j \big( \lambda(F+\rho M)+\lambda^2(G+\rho N) \big)
-\frac{1}{2}\rho (L+2\lambda M+ \lambda^2 N),
\]
and
\[
Y_d= -\lambda i \big( \rho(LF-ME)+\lambda((EG-F^2)+\lambda(LG-MF)) \big) +
\lambda^2 j\big( \rho(NE-MF)+
\]
\[
+ (EG-F^2)+\lambda(NF-MG) \big)+ \frac{1}{2} \lambda \rho (2ME-LF+\lambda(NE+2MF-LG)+\lambda^2 NF).
\]

As $0 < \lambda M (E +2\lambda F + \lambda^2 G)$, we have that only the first system of inequalities,
determining a rectangle $R_d$, has solutions.
Thus, for the number of diagonal edges with an equilibrium point, we obtain
\[
Z_d = \frac{(E+2\lambda F + \lambda^2 G) \rho M}{\lambda (EG-F^2)|(\rho \kappa_1 + 1)(\rho \kappa_2 + 1)|}.
\]
Thus, for the sum of the numbers of the edges with an equilibrium point on them we obtain the estimate:
\[
Z_h + Z_v + Z_d = \frac{-(LN-2MF+NE)\rho}{(EG-F^2)|(\rho \kappa_1 + 1)(\rho \kappa_2 + 1)|} =
\frac{-(\kappa_1 + \kappa_2)\rho}{|(\rho \kappa_1 + 1)(\rho \kappa_2 + 1)|}.
\]

The three different types of edges can be observed in Parts (e0), (e1), (e2) and (e3) of Figure~\ref{fig:flocks}.

\emph{Case 2}, $M < 0$.
We need to examine the edges $[p_{i,j},p_{i+1,j}]$, $[p_{i,j},p_{i,j+1}]$
and $[p_{i+1,j},p_{i,j+1}]$.
We may apply the approach described in the previous case, and obtain similar formulas
for the numbers of horizontal, vertical and diagonal edges that contain an equilibrium point.
In this way, for the sum of the numbers of the edges with equilibrium points we obtain the same estimate.

\emph{Case 3}, $M=0$.
Similarly like in Part 2, we may apply the formulas from the $M > 0$ case if $E_+$ belongs to $P^n$,
and those from the $M<0$ case $E_-$ does.
Note that the inequalities that determine if $[p_{i,j},p_{i+1,j}]$ or $[p_{i,j},p_{i,j+1}]$
contain an equilibrium point degenerate into the same inequalities for $M=0$, no matter which formulas we apply.
Thus, for $Z_h$ and $Z_v$ we obtain the same quantities.
Furthermore, the inequalities that determine if $E_+$ contains an equilibrium point in Case 1, and those
that determine if $E_-$ contains such a point in Case 2, determine a planar figure of zero area if $M=0$,
and thus, if $M=0$, we have $Z_d = 0$.

To compute $\Err_N$, we may apply the approach used in Parts 1 and 2.
We remark that if $M > 0$, then $X_d = X_h+X_v$ and $Y_d = Y_h+Y_v$, and if $M < 0$, then
$X_d = X_h - X_v$ and $Y_d = Y_h - Y_v$.
\end{proof}

\pagebreak

\begin{figure}[here]
\includegraphics[width=\textwidth]{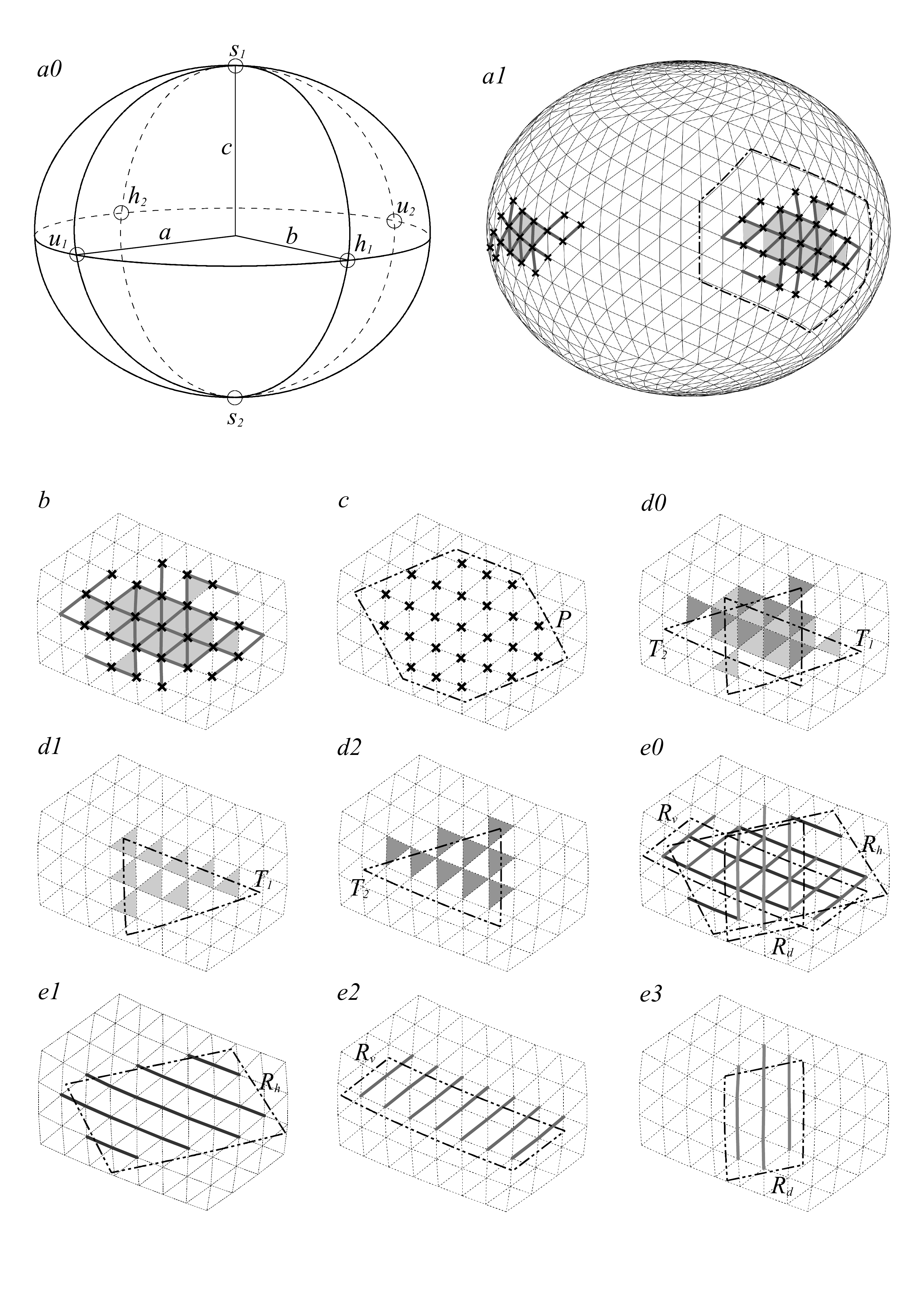}
\caption[]{Flocks of an ellipsoid}
\label{fig:flocks}
\end{figure}

\begin{myquote}{1cm}{1cm}
Figure~\ref{fig:flocks} shows an ellipsoid with axis ratios $a:b:c = 1.25 :1.15 : 1$.

\begin{itemize}
\item Part $a0$ shows the equilibrium points of the  ellipsoid. The stable, unstable and saddle points are denoted by $s_1$ and $s_2$, $u_1$ and $u_2$, and $h_1$ and $h_2$, respectively.
\item Part $a1$ shows the equilibrium points of $P^n$ near $u_1$ and $h_1$. Faces with a stable point are shaded, unstable vertices are marked with $\times$, and edges with a saddle point are drawn with bold lines.
\item Part $b$ shows the equilibrium points of $P^n$ near $h_1$. The zoomed hexagonal region is framed in Part a1.
\item Part $c$ shows the unstable equilibrium points near $n_1$ inside the hexagonal region $P$.
\item Part $d0$ shows the stable equilibrium points near $h_1$ with the triangles $T_1$ and $T_2$. These triangles are separately shown in Parts $d1$ and $d2$, respectively.
\item Part $e0$ shows the saddle type equilibrium points near $h_1$ with the paralelograms $R_1$, $R_2$ and $R_3$. Parts $e1$, $e2$ and $e3$ show these parallelograms separately.
\end{itemize}
The estimates in Theorem~\ref{thm:3Dmainapprox} yield $U^n(K)=20.19$, $S^n(K) = 22.59$ and $N^n(K)  = 43.78$ for the
`average' number of the imaginary equilibrium points in the `vicinity' of $h_1$. For the ellipsoid and the parametrization we used in Figure~\ref{fig:flocks}, we have $U^n(K) = 25$, $S^n(K) = 18$ and $N^n(K) =44$.
The error terms defined in Theorem~\ref{thm:3Dmainapprox} are the following: $\Err_U=15.11, \Err_S=30.04, \Err_N=44.98$. As we can see, the error functions grossly
overestimate the actual errors, as we pointed out after Theorem~\ref{thm:3Dmainapprox}.
\end{myquote}

\section{Results on the average of the equilibrium points of $P^n$}\label{sec:average}

In this section, for $x \in \Re$, $\fraction(x)$ denotes the fractional part of $x$.
We start with the following theorem.

\begin{thm}\label{thm:irrational}
We have the following.
\begin{enumerate}
\item[{\ref{thm:irrational}}.1] Let $I \subset [0,1]$ be a closed interval, and let $\eta$ be an irrational real number.
Let $A_n = \card \{ \fraction(k \eta) \in I : k=1,2,\ldots,n \}$.
Then $\lim\limits_{n \to \infty} \frac{A_n}{n}$ exists and is equal to the length of $I$.
\item[{\ref{thm:irrational}}.2] Let $X \subset [0,1]^2$ be a closed axis parallel rectangle, and
let $\eta_1, \eta_2$ be linearly independent irrational numbers in the vector space $\Re$ over $\Q$.
Let $B_n = \card \{ \big( \fraction(k \eta_1), \fraction(k\eta_2)\big) \in X : k=1,2,\ldots, n )\}$.
Then $\lim\limits_{n \to \infty} \frac{B_n}{n}$ exists and is equal to the area of $X$.
\end{enumerate}
\end{thm}

In the proof we use the following, well-known result about the simultaneous approximation of irrational numbers.
A more general statement is proven, for example, in Theorem 4.4, on page 45 of \cite{N56}.

\begin{lem}\label{lem:everywheredense}
Let $X$ be a closed axis parallel rectangle in the unit square $[0,1]^2$,
and let $\eta_1, \eta_2$ be linearly independent irrational numbers
in the vector space $\Re$ over $\Q$.
Then the set $\{ \big(\fraction(n \eta_1), \fraction(n \eta_2)\big) \in X : n=1,2,3, \ldots )\}$
is everywhere dense.
\end{lem}

\begin{proof}[Proof of Theorem~\ref{thm:irrational}]
We prove only (\ref{thm:irrational}.2), as the proof of (\ref{thm:irrational}.1)
is just its simplified version.
We set $S=[0,1]^2$ and $p_n = \big(\fraction(n \eta_1), \fraction(n \eta_2)\big)$.
For simplicity, the geometric figures in the proof are imagined as factored by $S$.

Let $a$ and $b$ denote the lengths of the horizontal and the vertical sides $X$, respectively.
By the previous lemma, for every $\delta > 0$, there is a positive integer $\bar{N}$
such that $\fraction(\bar{N} \eta_1) = \delta_1< \delta$ and
$\frac{\fraction(\bar{N} \eta_2)}{\fraction(\bar{N} \eta_1)} = \delta_2 < \delta$.
Note that the conditions of the theorem yield that $\delta_1$ and $\delta_2$ are irrational.

Let $M$ be the largest value with $\fraction(M\bar{N} \eta_2) < 1$.
Then
\[
\left( \frac{1}{\delta_1} -1  \right)\left( \frac{1}{\delta_2} -1  \right) \leq M \leq
\left( \frac{1}{\delta_1} +1  \right)\left( \frac{1}{\delta_2} +1  \right) .
\]

Consider the first $M\bar{N}$ points. We estimate how many of them are contained in $X$.
Pick any value $ 1 \leq k \leq \bar{N}$.
We examine the point set $P=\{ p_k, p_{k+\bar{N}}, p_{k+2\bar{N}}, \ldots, p_{k+(M-1)\bar{N}} \}$.
These points are equidistant points on a line ($\mod S$) with slope $\delta_2$ (cf. Figure~\ref{fig:average}).
We count the number of the points of $P$ contained in $X$.
Those that are are located on segments such that their projections on the $x$-axis, possibily apart from the first and the last one, are $a$ long.
Measured on the $x$-axis, the distance of the points of $P$ is $\delta_1$, and thus each such segment, possibly apart from the first and the last one,
contains at least $\frac{a}{\delta_1}-1$, and at most $\frac{a}{\delta_1}+1$ points.
There are at least $\frac{b}{\delta_2}-1$ `full' segments, and at most $\frac{b}{\delta_2}+1$ segments of any kind.
Hence, the total number of points in the rectangle is at least
$\left( \frac{a}{\delta_1}-1 \right) \left( \frac{b}{\delta_2}-1 \right)$ and at most
$\left( \frac{a}{\delta_1}+1 \right) \left( \frac{b}{\delta_2}+1 \right)$.

\begin{figure}[here]
\includegraphics[width=0.6\textwidth]{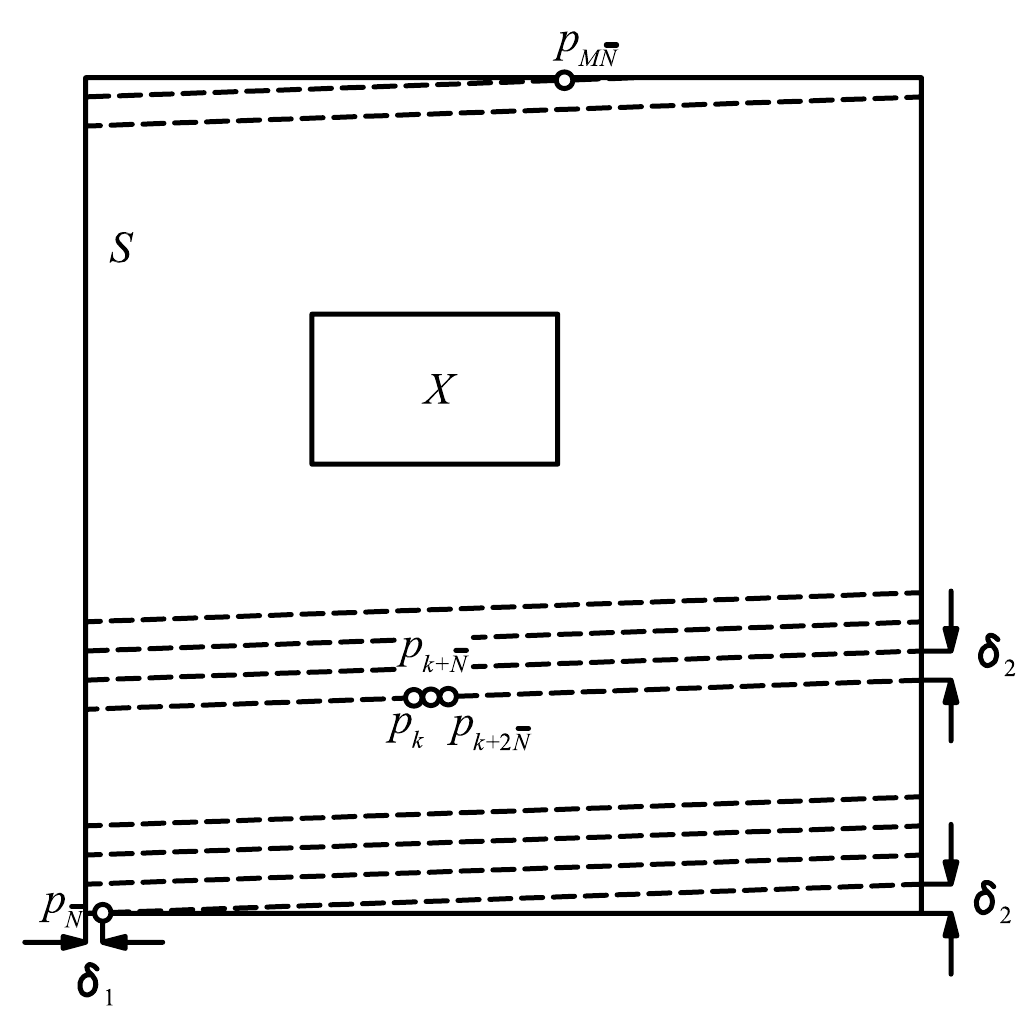}
\caption[]{An illustration for Theorem~\ref{thm:irrational}}
\label{fig:average}
\end{figure}

Thus, we have
\[
\frac{\bar{N} \left( \frac{a}{\delta_1}-1 \right) \left( \frac{b}{\delta_2}-1 \right)}{\bar{N}
\left( \frac{1}{\delta_1}+1 \right) \left( \frac{1}{\delta_2}+1 \right)} < \frac{B_{M\bar{N}}}{M\bar{N}} <
\frac{\bar{N} \left( \frac{a}{\delta_1}+1 \right) \left( \frac{b}{\delta_2}+1 \right)}{\bar{N}
\left( \frac{1}{\delta_1}-1 \right) \left( \frac{1}{\delta_2}-1 \right)}
\]
Observe that as $\delta_1, \delta_2 < \delta$, the inequalities above remain true if we replace
$\delta_1$ and $\delta_2$ by $\delta$.

Clearly, in the way described above we can prove the same inequalities for the second $M\bar{N}$ points, and so on.
Thus, for every value of $k$, we have
\[
\frac{(a-\delta)(b-\delta)}{(1+\delta)^2} < \frac{B_{kM\bar{N}}}{kM\bar{N}} < 
\frac{(a+\delta)(b+\delta)}{(1-\delta)^2} .
\]
Finally, if $kM\bar{N} \leq n < (k+1)M\bar{N}$, then
\[
\frac{B_n}{n} \leq \frac{B_{(k+1)M\bar{N}}}{kM\bar{N}} \leq \frac{(k+1)M\bar{N}}{kM\bar{N}}\frac{B_{(k+1)M\bar{N}}}{(k+1)M\bar{N}} \leq
\frac{n+\bar{N}M}{n-\bar{N}M} \frac{(a+\delta)(b+\delta)}{(1-\delta)^2},
\]
and we may obtain similarly that
\[
\frac{n-\bar{N}M}{n+\bar{N}M} \frac{(a-\delta)(b-\delta)}{(1+\delta)^2} \leq \frac{B_n}{n}.
\]

Now, let $\varepsilon > 0$ be arbitrary.
Then there is a sufficiently small $\delta > 0$ such that $\frac{(a+\delta)(b+\delta)}{(1-\delta)^2} \leq
ab + \frac{\varepsilon}{2} $.
Note that $M\bar{N}$ depends on $\delta$, but it is fixed if $\delta$ is fixed. Thus, for the value of $\delta>0$ above,
there is a positive integer $\bar{N}_1$ such that from $n > \bar{N}_1$ it follows that $\frac{n-\bar{N}M}{n+\bar{N}M} < \frac{ab+\varepsilon}{ab+\frac{\varepsilon}{2}}$,
and thus that $\frac{B_n}{n} < ab+\varepsilon$.
We obtain similarly the existence of a positive integer $\bar{N}_2$ such that $n > \bar{N}_2$ yields that
$\frac{B_n}{n} > ab-\varepsilon$.
Hence, with the notation $N^\star= \max \{ \bar{N}_1, \bar{N}_2 \}$, we have that $n > N^\star$ implies
that $| \frac{B_n}{n} - ab | < \varepsilon$, or in other words, that $\lim\limits_{n \to \infty} \frac{B_n}{n} = ab$.
\end{proof}

\begin{rem}\label{rem:Jordan}
The interval $I$ and the rectangle $X$ in Theorem~\ref{thm:irrational} can be replaced by any
Jordan measurable set.
\end{rem}

Theorem~\ref{thm:irrational} and Remark~\ref{rem:Jordan} yields the following corollary.

\begin{cor}\label{cor:almostthere}
We have the following.
\begin{enumerate}
\item[\ref{cor:almostthere}.1] Let $\eta \in \Re \setminus \Q$ and $X \subset \Re$ be a set
with Jordan measure $\lambda(X)$. Let $A_n = \card (\Z \cap (n \eta + X))$.
Then $\lim\limits_{n \to \infty} \frac{\sum\limits_{k=1}^n A_k }{n} = \lambda(X)$.
\item[\ref{cor:almostthere}.2] Let $\eta_1, \eta_2$ be linearly independent irrational numbers,
and let $X \subset \Eu^2$ be a set with Jordan measure $\lambda(X)$.
Let $B_n = \card \left( \Z^2 \cap ( n \eta_1 , n \eta_2 ) + X \right)$.
Then $\lim\limits_{n \to \infty} \frac{\sum\limits_{k=1}^n B_k}{n} = \lambda(X)$.
\end{enumerate}
\end{cor}

\begin{proof}
Again, we prove only (\ref{cor:almostthere}.2).
Clearly, we may assume that the diameter of $X$ is less than one and that it is contained in $[-1,0]^2$.
Then it may contain at most one point of the lattice $\Z^2$.
Observe that $( n \eta_1 , n \eta_2 ) + X$ contains a lattice point
if and only if $\big( \fraction(n \eta_1) , \fraction(n \eta_2 ) \big) + X$
contains the origin, which happens if and only if 
$\big( \fraction(n \eta_1) , \fraction(n \eta_2) \big)$ is contained in $-X$.
Since the measure of $-X$ is $\lambda(X)$, from this the assertion readily follows.
\end{proof}

In the previous sections we have seen that the numbers of the different types of equilibrium points 'fluctuate' around specific values
as we change $n$. Roughly speaking, in Theorem~\ref{thm:average} we prove that if the parameter of the equilibrium point
$m$ is positioned 'irregularly' in the parameter range, then the fluctuation 'averages out' in the long run.

\begin{thm}\label{thm:average}
We have the following.
\begin{enumerate}
\item[\ref{thm:average}.1]
Let $r : [\tau_1,\tau_1] \to \Re^2$ satisfy the conditions of Section~\ref{sec:2D}.
Assume that $\frac{\tau_1}{\tau_2-\tau_1}$ is irrational.
Then, using the notations of Section~\ref{sec:2D}, if $K$ is sufficiently large, we have
\begin{eqnarray*}
\lim_{n \to \infty} \frac{\sum\limits_{k=1}^n U^k(K)}{n} & = & \frac{\rho | \kappa |}{|\rho \kappa + 1|} \quad \mathrm{and}\\
\lim_{n \to \infty} \frac{\sum\limits_{k=1}^n S^k(K)}{n} & = & \frac{1}{|\rho \kappa + 1|}.
\end{eqnarray*}
\item[\ref{thm:average}.2]
Let $r : [u_1,u_2] \times [v_1,v_2] \to \Re^3$ satisfy the conditions of Section~\ref{sec:3D}.
Assume that $\frac{u_1}{u_2-u_1}$ and $\frac{v_1}{v_2-v_1}$ are linearly independent irrational numbers.
Then, using the notations of Section~\ref{sec:3D}, if $K$ is sufficiently large, we have
\begin{eqnarray*}
\lim_{n \to \infty} \frac{\sum\limits_{k=1}^n U^k(K)}{n} & = &
\frac{\kappa_1 \kappa_2 \rho^2}{|(\kappa_1 \rho + 1)( \kappa_2 \rho + 1)|},\\
\lim_{n \to \infty} \frac{\sum\limits_{k=1}^n N^k(K)}{n} & = &
\frac{-(\kappa_1 + \kappa_2) \rho}{|(\kappa_1 \rho + 1)( \kappa_2 \rho + 1)|} \quad \mathrm{and}\\
\lim_{n \to \infty} \frac{\sum\limits_{k=1}^n S^k(K)}{n} & = & \frac{1}{|(\kappa_1 \rho + 1)( \kappa_2 \rho + 1)|}.
\end{eqnarray*}
\end{enumerate}
\end{thm}

\begin{proof}
We prove only (\ref{thm:average}.2).
Consider, for example, Case 1 of Part 2; that is, the case of $S^n(K)$ with $M > 0$.
In the proof, we estimated the number of points $p_{i,j}$ such that $(i,j)$ is contained 
the triangles $T^{-1}(T_1)$ and $T^{-1}(T_2)$.
Let $\eta_1 = \frac{\bar{u}-u_1}{u_2-u_1}$ and $\eta_2 = \frac{\bar{v}-v_1}{v_2-v_1}$
With the notation used in the proof,
we have that $\fraction(i) = \fraction(n \eta_1)$ and $\fraction(j) = \fraction(n \eta_2)$.
Thus, we need to determine the number of the points of the lattice
$(n \eta_1, n \eta_2 ) + \Z^2$ in $T_1$ or in $T_2$.
Hence the assertion follows from the previous corollary.
In the rest of the cases, we may apply a similar argument.
\end{proof}

\section{Questions and concluding remarks}\label{sec:remarks}

In light of our results in Sections~\ref{sec:2D} and \ref{sec:3D}, we introduce the following notions.

\begin{defn}\label{defn:imaginary_index}
Let $m \in \Re^d$, where $d=2$ or $d=3$, be a generic equilibrium point of a smooth hypersurface $H \subset \Re^d$ with respect to $o$, and set $\rho = ||m||$.
Then the \emph{imaginary equilibrium indices} of $m$ are the quantities:
\begin{itemize}
\item $U^\star = \frac{|\rho \kappa |}{|\rho \kappa + 1|}$ and $S^\star = \frac{1}{|\rho \kappa + 1|}$ for $d=2$, where $\kappa$ is the (signed) curvature of $H$ at $m$;
\item $U^\star = \frac{|\rho^2 \kappa_1 \kappa_2 |}{|(\rho \kappa_1 + 1)(\rho \kappa_2 + 1)|}$, $N^\star = \frac{|\rho (\kappa_1 + \kappa_2 )|}{|(\rho \kappa_1 + 1)(\rho \kappa_2 + 1)|}$
and $S^\star = \frac{1}{|(\rho \kappa_1 + 1)(\rho \kappa_2 + 1)|}$ for $d=3$, where $\kappa_1$ and $\kappa_2$ are the fundamental curvatures of $H$ at $m$.
\end{itemize}
\end{defn}

In our investigation, we found the equilibrium points of $P^n$ with indices
in a given `large' neighborhood of $0$ or $(0,0)$.
It would be convenient to say that if we choose a large value of $K$, then, for a fine discretization, we have found all the
equilibrium points of $P^n$ in this way.
Unfortunately, we cannot do that: our method does not take into account an infinite sequence $p_{i_k}^{n_k}$ of equilibria if $i_k$ is not a bounded sequence
of $n_k$. This happens if, for example, $i_k$ is of order $\sqrt{n_k}$, in which case $p_{i_k}^{n_k} \to m$ is still satisfied.
This observation leads to the following definition.

\begin{defn}
Let $r$ be a curve satisfying the conditions in Section~\ref{sec:2D}, with a unique equilibrium point $m=r(0)$.
If $\{ p_{i_k}^{n_k} \}$ is a sequence of equilibrium points of $P^n$ with
$\lim\limits_{k \to \infty} i_k = \infty$, then the sequence $\{ p_{i_k}^{n_k} \}$ is called an \emph{irregular} equlibrium sequence.
\end{defn}

We may similarly define the notion of an irregular equilibrium sequence for a surface.
Note that if a curve or surface has no irregular equilibrium sequence, then for this curve
or surface the estimates for $U^n(K)$, $N^n(K)$ and $S^n(K)$ in Sections~\ref{sec:2D} and \ref{sec:3D}
hold for the numbers of the different types of all the equilibrium points of $P^n$.

A direct computation shows that if $r : [\tau_1,\tau_2] \to \Re^2$ satisfies the conditions in Section~\ref{sec:2D},
and its coordinate functions are polynomials, then $r$ has no irregular sequences.
This leads to the following question.

\begin{ques}
Prove or disprove that if the coordinate functions of $r$ are analytic functions of $\tau$,
then the curve has no irregular equilibrium sequence. What about surfaces in $\Re^3$?
\end{ques}

Our method was to approximate our smooth surface $S$ with a polyhedral, and thus nonsmooth, surface $P^n$.
We found that no matter how close $P^n$ is to $S$, it may occur that $P^n$ has more equilibrium points than $S$.
This is not the case if we approximate $S$ with a surface the first and second partial derivatives of which
are close to those of $S$.

\begin{rem}
Let $r:D \to \Re^3$ be a $C^2$-class surface satisfying the conditions of Section~\ref{sec:3D}.
Let $\varepsilon>0$ be sufficiently small. Then any $g:D \to \Re^3$ $C^2$-class surface
with the property that the difference between $r(u,v)$ and $g(u,v)$, and between any first or second partial
derivatives of $r$ and $g$ at $(u,v) \in D$ is less, than $\varepsilon$, has exactly one equilibrium $m'$.
Furthermore, $m'$ is stable, unstable or a saddle point if, and only if, $m$ is stable, unstable or a saddle point,
respectively.
\end{rem}

\begin{proof}
We use the notations from Section~\ref{sec:3D}.

Let $r : D \to \Re^3$ be a $C^2$-class surface as in Section~\ref{sec:3D}.
Note that this surface has a stable, unstable or saddle point at $m=r(0,0)$ if exactly two, zero and one of $\rho \kappa_1 + 1$ and $\rho \kappa_2 +1$ is positive.
Thus, since $r$ is $C^2$, if we approximate $r$ by $g$ in a way that the surface, and its first and second partial derivatives change at most $\varepsilon$ for some small $\varepsilon > 0$,
then the signs of $\rho \kappa_1 + 1$ and $\rho \kappa_2 +1$ do not change in a neighborhood of $g(0,0)$.
Clearly, if $\varepsilon$ is sufficiently small, $g$ has no equilibrium outside this neighborhood, and the type of any equilibrium inside
is the same as that of $m$.
Hence, our remark follows from the Poincar\'e-Hopf theorem.
\end{proof}

\begin{rem}
In Section~\ref{sec:3D}, to prove Theorem~\ref{thm:3Dmainapprox}
we made the assumptions that $\lambda |M| \leq |L|$ and $\lambda |M| \leq \lambda^2 |N|$.
Nevertheless, in the opposite case, a slight modification of the proof of the theorem shows
that the estimates in the theorem are, apart from the error terms, lower
bounds for $U^n(K)$, $N^n(K)$ and $S^n(K)$.
\end{rem}

\begin{rem}\label{rem:errorterms}
We have not examined how large the error terms in Theorem~\ref{thm:3Dmainapprox} are compared to the estimates
of $U^n(K)$, $N^n(K)$ and $S^n(K)$. Note that the estimates are approximated by the areas of certain convex
regions, and that the error terms are approximated by the sum of the half perimeters of certain parallelograms
circumscribing these regions. Thus, roughly speaking, if the estimates are `large',
then 'probably' the error terms are `relatively small'.
\end{rem}

\begin{rem}
In Theorem~\ref{thm:average}, we described the number of equilibrium points in a more precise way than in Theorems~\ref{thm:main} and \ref{thm:3Dmainapprox},
under a special condition determined by the location of the parameters of the equilibrium point of $r$ in the domain of $r$.
Note that choosing the parameter range `randomly', the conditions of Theorem~\ref{thm:average} are satisfied with probability one.
\end{rem}

\begin{rem}
Although we deal only with generic equilibrium points, Theorem~\ref{thm:3Dmainapprox} also gives a prediction for the degenerated cases as $\kappa \rho \to -1$.
Assuming $\rho$=1 and $\kappa_1<\kappa_2$,  as $\kappa_1 \to -\infty$ and $\kappa_2 \to -1$ we obtain the following limits:
\bea \label{lim1}
\lim\limits_{\kappa_1 \to -\infty} \lim\limits_{\kappa_2 \to -1} U^\star    =   \lim\limits_{\kappa_2 \to -1} \lim\limits_{\kappa_1 \to -\infty} U^\star  & = &  \infty \\
\label{lim2}
\lim\limits_{\kappa_1 \to -\infty} \lim\limits_{\kappa_2 \to -1} N^\star  = \lim\limits_{\kappa_2 \to -1} \lim\limits_{\kappa_1 \to -\infty} N^\star & = & \infty,
\eea
however, the limit of $S^\star$ is undefined:
\bea
\label{lim3}
\lim\limits_{\kappa_2 \to -1} \lim\limits_{\kappa_1 \to -\infty} S^\star  & =  & 0, \\
\label{lim4}
\lim\limits_{\kappa_1 \to -\infty} \lim\limits_{\kappa_2 \to -1} S^\star   & =  & \infty .
\eea
Degenerated flocks corresponding to this limit are illustrated on a flattened oblate spheroid in Parts (a) and (b) of Figure~\ref{fig:disc-rod}.
In accordance with (\ref{lim1})-(\ref{lim2}) we can observe large numbers of  unstable and saddle points, however, the number of stable equilibrium points 
may be low (cf. Equation (\ref{lim3}) and Part (a) of Figure~\ref{fig:disc-rod}) or high (cf. Equation (\ref{lim4}) and Part (b) of Figure~\ref{fig:disc-rod}).
Similarly, as $\kappa_1 \to -1$ and $\kappa_2 \to 0$, $S^\star$ and $N^\star$ converges to $\infty$, however, the limit of $U^\star$ is undefined. This situation is illustrated on an elongated prolate spheroid in Parts (c) and (d) of Figure~\ref{fig:disc-rod}.
\end{rem}

\begin{figure}[here]
\includegraphics[width=\textwidth]{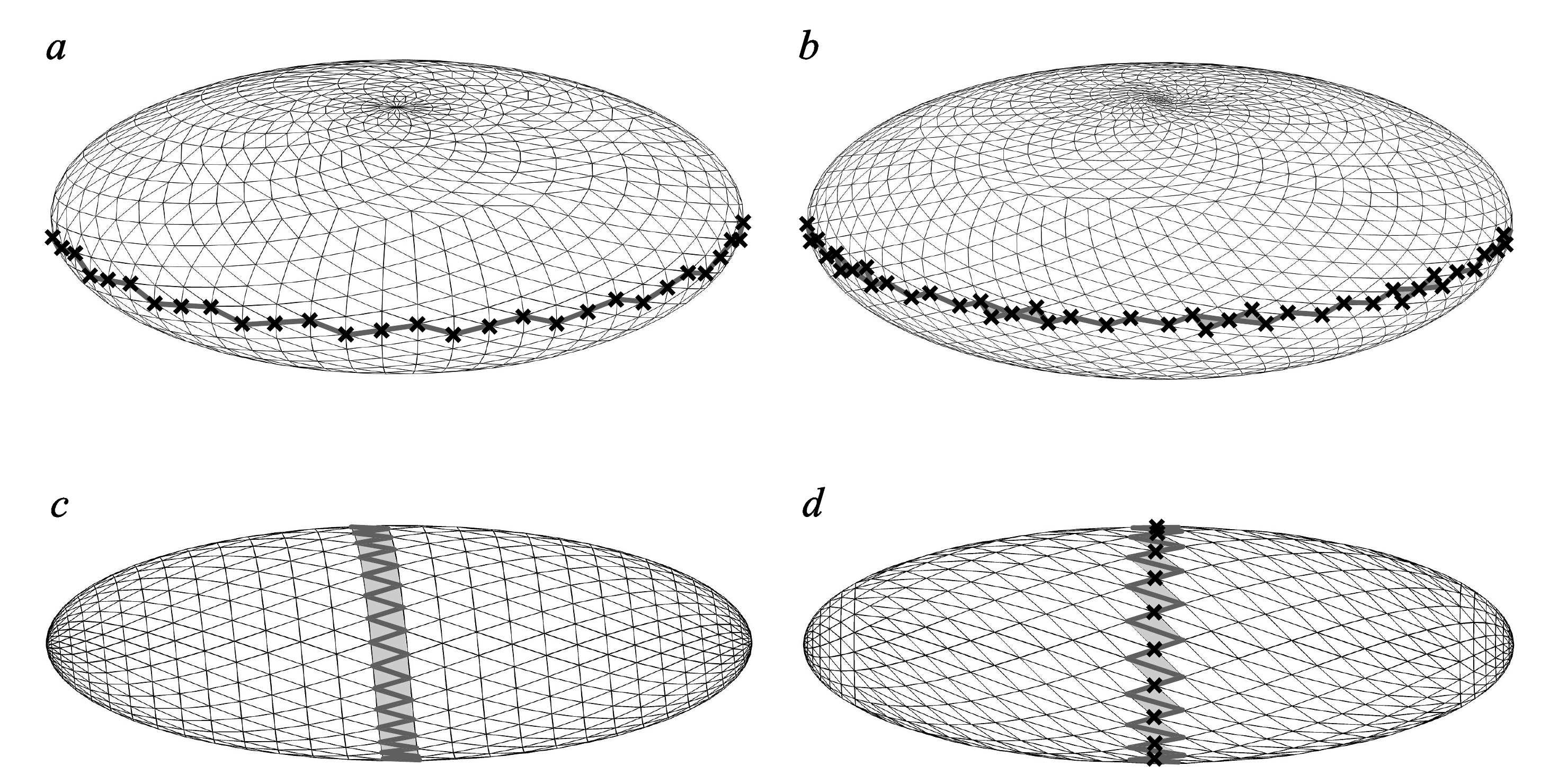}
\caption[]{Flocks on an oblate and on a prolate spheroid}
\label{fig:disc-rod}
\end{figure}

\begin{rem}\label{pebble}
Imaginary equilibrium indices were computed on natural pebble surfaces and were compared with the real number of equilibrium points (Figure~\ref{fig:pebble}). Pebbles were digitized by a high-accuracy 3D scanning method, resulting a dense point cloud. Since equilibrium points are located on the convex hull of the surface, we identified equilibria on the convex hull of the point cloud \cite{Domokos2}. Imaginary equilibrium indices were computed by using a curvature estimation method \cite{Chen} on the triangular mesh. Although the convex hull of a pebble is clearly not an equidistant discretization, as Figure~\ref{fig:pebble} illustrates, imaginary equilibrium indices give a fair estimate on the actual (integer) number of equilibrium points.
\end{rem}

\begin{figure}
\includegraphics[width=\textwidth]{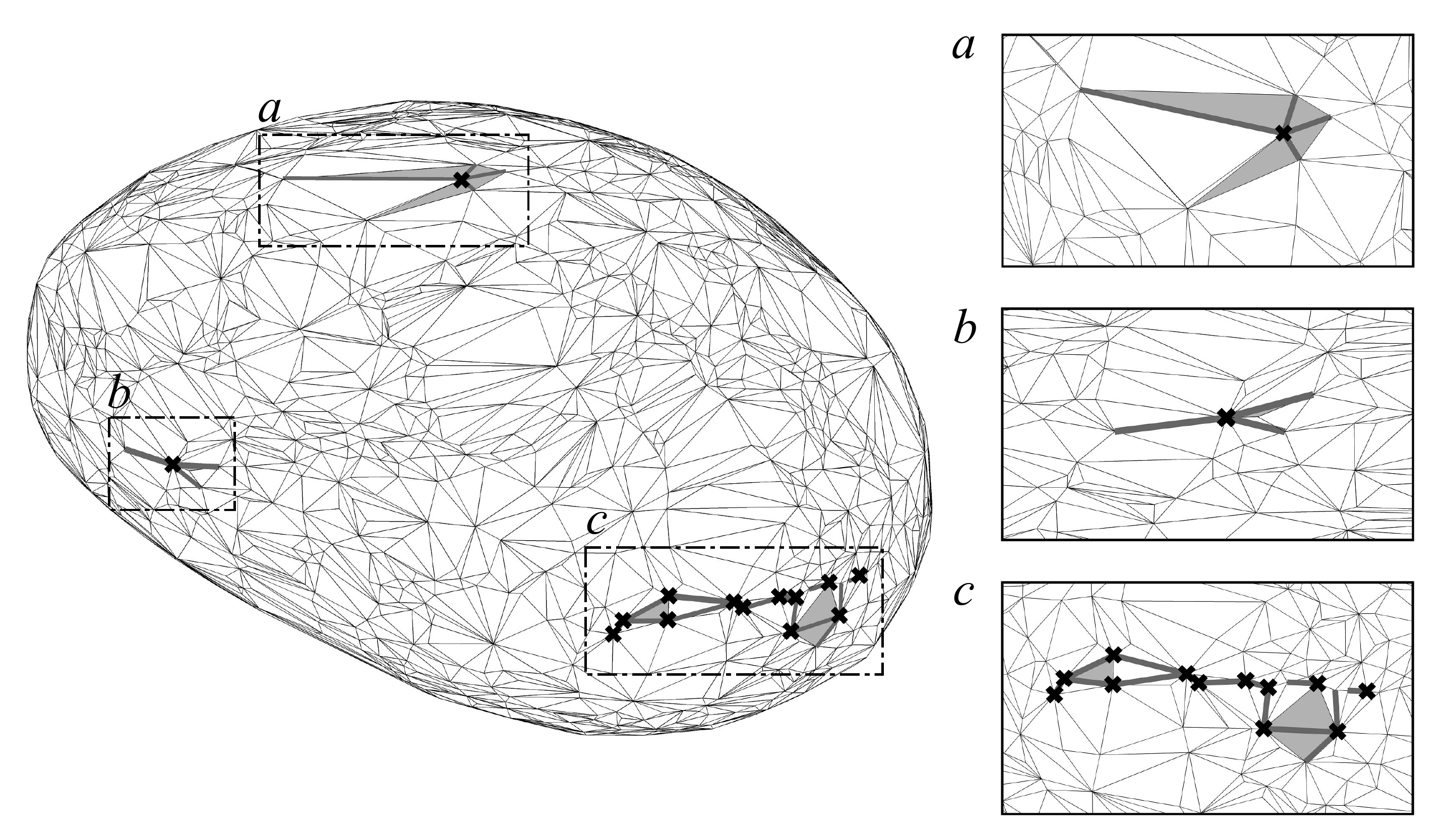}
\caption[]{Flocks on a pebble.
Imaginary equilibrium indices ($S^\star,U^\star,N^\star$) predicted by ÿ(\ref{mainresult}) and actual numbers of equilibrium points ($S,U,N$) measured on the pebble:
\begin{tabular}{|c||c|c||c|c||c|c|}
\hline
flock & $S^\star$ & $S$ & $U^\star$ & $U$ & $N^\star$ & $N$ \\
\hline 
\hline
a & 4.84 & 4 & 0.60 & 1 & 4.44 & 4 \\
\hline
b & 0.72 & 1 & 1.03 & 1 & 2.75 & 3 \\
\hline
c & 2.88 & 3 & 9.55 & 12 & 11.43 & 14 \\
\hline
\end{tabular} }
\label{fig:pebble}
\end{figure}

Clearly, imaginary equilibrium indices can be defined for hypersurfaces $H \subset \Re^d$ for $d>3$ as well.
Hence, a natural question is to ask about higher dimensional analogues of our theorems.

\section{Acknowledgement}
The results discussed above  are supported by the Hungarian Research Fund (OTKA) grant 72146  and grant T\'AMOP -
4.2.2.B-10/1--2010-0009.
The authors are indepted to M\'arton Nasz\'odi for his helpful remarks regarding the results in Section~\ref{sec:average}.

\end{document}